\documentclass{article}
 \newcommand {\theoremstyle} [1] { }
 \newenvironment{proof}{{\noindent\it\underline{Proof}}:}{\hfill$\Box$}
  \newenvironment{pot}{{\noindent\it\underline{Proof of Theorem \ref{Gronwall abstracto}}}:}{\hfill$\Box$}
    \newenvironment{pot2}{{\noindent\it\underline{Proof of Theorem \ref{main}}}:}{\hfill$\Box$}
 \newtheorem{thm}{Theorem}[section]
 \newtheorem{prop}{Proposition}[section]
 \theoremstyle{plain}
 \newtheorem{lem}[thm]{Lemma} %%Delete [thm] to re-start numbering
 \newtheorem{cor}{Corollary}[section]
 \theoremstyle{definition}
 \newtheorem{defn}[thm]{Definition}
 
 \theoremstyle{remark}
 \newtheorem{rem}[thm]{Remark}

\usepackage {amssymb}
\usepackage{graphicx}
\usepackage{mathtools}
\usepackage{xfrac}
\usepackage{xcolor}

\def\R{\mathbb{R}}

\usepackage {amssymb}

\makeatother
\title{An abstract Gronwall inequality on a Banach lattice}
\date{}
\author{P. Amster and J. Epstein}

\begin{document}

\maketitle

	\begin{center}
 Departamento de Matem\'atica\\ Facultad de Ciencias Exactas y Naturales, Universidad de Buenos Aires\\ \& IMAS-CONICET\\ 
		Ciudad Universitaria - Pabell\'on I, 1428, Buenos Aires, Argentina \\
		{\textbf{E-mail}: pamster@dm.uba.ar -- jepstein@dm.uba.ar}
		\vspace{0.2cm}



	\end{center}
\begin{abstract} 
%May the Gronwall Lemma be regarded as a maximum principle? In this expository paper, 
An abstract version of the celebrated inequality is described by means of the spectral bound of 
an  operator defined on a Banach lattice. As a consequence, uniqueness and continuous dependence results for the general semilinear problem $Lu=N(u)$ are established and a connection with the maximum principle is explored. 
     
\end{abstract}

\section{Introduction}

   The Gronwall-Bellman inequality, in its simplest formulation, can be stated as follows: 
 
 \begin{lem} \label{Gronwall clasico}
Assume that  $x\in C[a,b]$  satisfies
$$ x(t)\leq A+B\int_{a}^{t}x(s)ds$$
for all $t\in [a,b]$ for some constants $A\in\R$ and $B\ge 0$.  Then 
$$x(t)\leq Ae^{B(t-a)} \qquad t\in [a, b].$$\end{lem}

\noindent As it is known, this result is crucial  in order to prove the uniqueness and continuous dependence of the solutions for the initial value problem 
\begin{equation}
    \label{ivp}
    x'(t)=f(t,x(t)), \qquad x(a)=x_0 
\end{equation}  
for $f:U\subset \R^{n+1}\to\R^n$ continuous and locally Lipschitz with respect to its second variable. 
Originally established in \cite{gro}, several extensions 
of the preceding lemma have been obtained. 
An early survey of results can be found in   
\cite{cha} and the references therein. Various generalizations 
include discrete analogues,
nonlinear cases, 
inequalities in partially 
ordered Banach Spaces, Stieltjes integrals and applications to singular, partial, fractional differential equations, 
among others, see e.g. \cite{AT,cha, drag, erbe, gall, ming,  ro, ye}. 

The purpose of this paper is to formulate an abstract version of the celebrated inequality and apply it  to a 
 general problem of the form
\begin{equation}\label{L=N}
Lx=N(x),\qquad \mathcal P x=x_0.
\end{equation}
Here,  $L:D\to X$ is a linear operator, where $X$ is a normed space, $D\subset X$ is a dense subspace, $N:X\to X$ is continuous  and $\mathcal P:D\to D$ is a  projector onto the kernel of $L$. 
The fact that the initial value problem is a particular case of (\ref{L=N})  becomes  
obvious 
after
setting $X:=C[a,b]$, $D:=C^1[a,b]$, 
$Lx:=x'$ and $N(x)(t):=f(t,x(t))$. 
Here, the role of $\mathcal P$ is understood 
if we observe, in the first place, that the kernel of $L$ is the subspace of constant functions, which can be identified with $\R^n$ and, in the second place, that the integral form of 
 (\ref{ivp}) is  readily written   as a functional equation
\begin{equation}\label{abst-inteq}
x= x_0 + KN(x),     
\end{equation}
where $K\varphi(t):=\int_{a}^t \varphi(s)\, ds.$  It is noticed that the original problem is recovered by applying $L$ at both sides of the equality, which yields the identity $Lx=N(x)$ and, in turn, replacing again in (\ref{abst-inteq}) this means that $x= x_0 + KLx$. 
With this in mind, for the abstract problem (\ref{L=N}) we shall always assume that $L$ has a (not necessarily compact) right inverse $K:X\to X$ and  that $\mathcal Px:= x - KLx$.

With the aim of introducing a general abstract version 
of the Gronwall lemma, let us start by noticing that the function $y(t):= Ae^{B(t-a)}$ satisfies
$$y(t)=A+B\int_a^t y(s)ds. 
$$
Thus, if $x$
verifies the assumption of the previous lemma, then taking $z(t):= x(t)-y(t)$ it is obtained 
\begin{equation}
    \label{gron-weak}z(t)\le B\int_a^t z(s)ds
\end{equation}
for all $t\in [a,b]$. Next, observe that the statement of Lemma \ref{Gronwall clasico} simply says that $x\le y$; 
in other words, the result  is in fact equivalent to its  particular case $A=0$, that is:  

 \begin{lem} \label{Gronwall cero}
Assume that  $z\in C[a,b]$  satisfies
(\ref{gron-weak})
for all $t\in [a,b]$ and  some constant $B\ge 0$.  Then $z(t)\le 0$ for all $t\in [a,b]$.
\end{lem}

An abstract version of the preceding result 
can be formulated within the context of Banach lattices which, roughly speaking,
are ordered Banach spaces with  binary supremum and infimum functions. A more detailed definition shall be given in Section \ref{resolv}. 

The following theorem is an immediate consequence of known results; the contribution of the present work 
is just
to propose an interpretation as a generalized Gronwall-Bellman inequality and to explore some of its implications. 
 
\begin{thm}\label{Gronwall abstracto}
Let  $X$ be a real Banach lattice and let $K\in \mathcal L(X)$ be positive, that is, order preserving. Let $\rho_K$ be the {spectral bound} of $K$, defined as 
$$\rho_{K}:=\sup\lbrace \Re( \lambda ), \ \lambda\in \sigma(K)\rbrace.$$
If $sz\le Kz $ for some 
 $s>\rho_K$, then $z\le 0$. 
  \end{thm}

As is well known, when $K$ is compact and  
the dimension of $X$ is infinite, it is verified that
 $0\in \sigma(K)$ and consequently   $\rho_K\ge 0$. 
In particular, if we consider
the Banach lattice   $X=C[a,b]$ with the usual pointwise order  
 and the 
 positive compact operator $K:X\to X$ given by $Kx(t):=\int_{a}^tx(s)ds$, then  $K$ is quasinilpotent, 
 i.e.   $\sigma(K)=\{0\}$ which, in turn, yields  $\rho_{K}=0$. 
 Thus, we are in condition of deducing 
 Lemma \ref{Gronwall cero} from Theorem \ref{Gronwall abstracto} in a straightforward 
manner. Indeed, since the case $B=0$ is trivial, we may assume $B>0$ and set $s:=\frac{1}{B}>\rho_K$; hence, the result follows because 
 (\ref{gron-weak}) 
simply says that  $ sz\le Kz.$

Furthermore, exactly as described before, 
a somewhat more general version 
is readily deduced from the obvious fact that
if $s>\rho_K$ then $sI-K$ is invertible. In particular, for arbitrary $A\in X$ we may set
$y=y(s,A)$ as the unique element of $X$ such that  $sy=A + Ky$, namely $y=(sI-K)^{-1}A$. 
Thus, the following corollary may be regarded as a general Gronwall lemma in its standard formulation:

\begin{cor} \label{y(s,A)} In the previous context, for any $s>\rho_K$ 
     and $A\in X$ it    holds that 
\begin{equation}\label{standard}
sx \le A + Kx \Longrightarrow x \le y(s,A).    
\end{equation}

\end{cor}

 Interestingly, most known (linear) versions of the  
 Gronwall Lemma involve quasinilpotent operators, 
 so Theorem \ref{Gronwall abstracto}
may be seen as a  unifying approach; 
however, the standard form expressed in (\ref{standard}) usually requires an explicit 
computation of the function $y$, which might 
be a difficult task. 
We observe that Corollary 
\ref{y(s,A)} 
is sharp, in the sense that if the inequality 
$sx\le A+Kx$ implies $x\le z$ for 
some $z$, then $y(s,A)\le z$. 
This is just a trivial consequence of  the 
definition of $y(s,A)$; however, it could  become a useful tool to establish the sharpness of 
some known estimates. 
Some discussion on this topic shall be tackled in Section 
\ref{extensions}. 

As a  consequence of the 
preceding results, we  establish the 
following maximum principle, which can be applied to 
problems of the form (\ref{L=N}):

\begin{thm}\label{main} 
Let $D$ be a dense subspace of a Banach lattice $X$, assume that $L:D\to X$ is  a linear operator with a positive  
right inverse $K\in \mathcal L(X)$ and define  $\mathcal Px:=x- KLx$. Fix $B\ge 0$ such that  $B\rho_K<1$. If $x\in D$ satisfies 
\begin{align}\label{inecuacion 1}
Lx &\leq B x,\\
\mathcal{P}x&\leq 0 
\end{align}
then $x\leq 0$.
\end{thm}

 As it may be observed, 
even in the case in which $\mathcal P$ is continuous, the (onto) mapping $L$ is not necessarily a Fredholm operator.
Specifically, it may happen that $\ker(L)$ is infinite-dimensional: for example,  consider $X=C(\overline \Omega)$, where $\Omega\subset \R^n$ is a smooth bounded domain, and 
the 
operator $L:=-\Delta$,  
defined on a dense subspace of $X$. 
In spite of that, Theorem \ref{main}
 can  still be applied and allows  
 to give a simple 
 interpretation of the well-known maximum principle for the Laplacian operator. This connection shall be investigated in Section \ref{lapl}.
%In more precise terms, the standard (weak) maximum principle ensures that, if $Lx\le 0$ and $x\le 0$ over $\partial\Omega$, then $x\le 0$ in $\Omega$. 
%However, this can be extended up to the first eigenvalue of $L$ under Dirichlet boundary conditions; namely, for $B<\lambda_1$
%$$Lx \le Bx,\quad x|_{\partial\Omega}\le 0 \Rightarrow x\le 0. $$
%This latter result ant its connection with Lemma \ref{Gronwall clasico} will be the subject of Section  \ref{lapl}. 

 The paper is organized as follows. In the next section, for the sake of completeness, we recall the basic definitions and  a well known integral formula for the resolvent of a positive bounded operator, which yields to  
 immediate proofs of Theorems \ref{Gronwall abstracto} and
 \ref{main}, carried out in   Section 
 \ref{dem-main}. 
 Applications to the standard continuous and  discrete cases are shown in Section \ref{discrete}, in which some novel proofs of known results are offered. 
 In Section \ref{lapl}, 
 we  present the maximum principle for a Laplacian operator, whose kernel is infinite-dimensional, as a particular application of our abstract Gronwall lemma.  
In Section \ref{uniq}, abstract versions 
 of uniqueness and continuous dependence results
 derived from Gronwall lemma are presented. In particular, under an appropriate local Lipschitz assumption it shall be shown  
 that the set of values $x_0\in\ker(L)$ for which a solution of (\ref{L=N}) exists is open. 
 Some  final remarks and comments are included 
 in Section \ref{disc}.

\section{An integral formula for the resolvent }
\label{resolv}

In this section, we recall some  basic facts concerning the integration of functions of a real variable in a Banach space $X$ 
and the Banach lattice structure. 
Most proofs are included with the aim of providing a self-contained exposition. 
 Throughout the section, the concept of integral for a   bounded function $f:[a,b]\to X$ shall be understood in the Riemann sense, that is: $f$ is integrable if the limit 
$$\int_a^b f(t)dt:=\lim_{|\pi|\to 0 } \sum_{i=0}^{n-1} f(c^\pi_i)(t_{i+1}-t_i)$$ exists, where  
$\pi=\{a=t_0<t_1<...<t_n=b\}$ is a partition with norm   
$|\pi|:=\max\{t_{j+1}-t_j\}$ and $c_i^\pi\in [t_i,t_{i+1}]$ is arbitrary for each $i$.  
 
\begin{defn}
A Banach lattice is a Banach space $X$ equipped with a compatible partial order $\le$, 
a  supremum function $\vee:X\times X\to X$ satisfying
$$x, y \le x\vee y,\qquad  x\vee y\le z \quad \hbox{for all $z\ge x,y$}
$$
and such that the absolute value 
%$|x|\le \|x\|$ for all $x$, 
defined by 
$$|x|:= x\vee -x
$$
verifies 
\begin{equation}\label{normal}
|x|\le |y|\Longrightarrow \|x\|\le \|y\|.    
\end{equation}
\end{defn}
\begin{rem}
It is observed that no mention of the infimum is required in the preceding definition,
since it can be defined by De Morgan's law, namely
$$x\wedge y:= -(-x\vee -y).
$$
\end{rem}

\begin{prop}\label{modulo igual a cero}
The following properties hold:

\begin{enumerate}
    \item $|x|\ge 0$. Consequently, $|x|\leq 0$ implies $x=0$.
    \item $|x+y|\le |x| + |y|$.
\end{enumerate}
\end{prop}
\begin{proof}
Because 
$|x|=x\vee -x \ge \pm x$, it follows that $2|x|=|x|+|x|\ge 0$ and the first property follows.  For the second one, simply observe that, since $|x|\ge \pm x$ and $|y|\ge \pm y$, it is deduced that 
$|x|+|y|\ge \pm (x+y)$. 
\end{proof}

\begin{prop} \label{equiv-norm} The norm of $|x|$ coincides with the norm of $x$. In particular, $x\to 0$ if and only if $|x|\to 0$. 
    
\end{prop}

\begin{proof}
    Let  $y:=|x|$  and observe, on the one hand,  that $|x|\le |y|$, so $\|x\|\le \|y\|$. 
    On the other hand, we know from the previous proposition that  $y\ge 0\ge -y$, so $|y|=y=|x|$, 
    whence  $\|y\|\le \|x\|$. 
\end{proof}

\medskip 

\noindent It is clear from the previous definitions that if 
$f:[a,b]\to X$ is an integrable function such that $f(t)\geq 0$ for all $t\in[a,b]$, then $\int_a^b f(t)dt\geq 0$.

We recall that an operator $M\in \mathcal L(X)$ is called \textit{positive} when it preserves order. Due to linearity, this is equivalent to saying that $Mx\ge 0$ for all $x\ge 0$. 
It is worth mentioning that,
in the specific context of Hilbert spaces, there is another notion of positive operator. 
Namely, when $X$ is a Hilbert space, it is  said that $M\in \mathcal L(M)$  is positive if $\langle Mx, x\rangle\geq 0$ for all $x\in X$ (in real Hilbert spaces, sometimes a self-adjointness assumption is also imposed). 
When both the lattice and Hilbert structures are present, the order preserving property may be related with the positiveness in the Hilbert sense, provided that certain compatibility between the cone of positive elements with the inner product is satisfied. However
the two notions are, in general, independent.

In order to complete the machinery 
needed for our results,   
 let us also recall that, for $M\in \mathcal L(X)$,  the $C_0$-semigroup $T_M:[0,+\infty)\to \mathcal L(X)$  generated by $M$
 is defined as the operator exponential function given by

%\begin{defn}
%Given a  Banach lattice $X$, we shall say that  $T:\mathbb{R}_+ \to \mathcal L(X)$ is a $C_0$-semigroup if
%\begin{enumerate}
%\item $T(0)=\mathrm{id}_X$.
%\item $T(s+t)=T(s)T(t)$ for all $t,s\geq 0$.
%\item $\lim_{t\to 0^+}\|T(t)x-x\|=0$ for all $x\in X$.
%\end{enumerate}
%Moreover, we shall say that  $T$ is positive, if $T(t)$ is a positive operator for all $t\geq 0$.
%\end{defn}

$$T_M(t)=e^{tM}:=\sum_{k=0}^\infty \frac{t^kM^k}{k!}.$$
If furthermore $M$ is positive, then  $T_M$ is positive, although  the converse is not true: observe, for instance, that  $e^{-tI}=e^{-t}I$. The following result can be found in  \cite[Ch. C-III]{Na}.
\begin{prop}\label{resolvent} 
Let $M\in \mathcal L(X)$ be positive %, let $T_A$ be the $C_0$-semigroup generated by $A$ 
and let $\rho_M$ be the  spectral bound of $M$. Then, for every $x\in X$ and $\lambda\in \mathbb{C}$ such that $\Re(\lambda)>\rho_M$ we have the following identity: 
$$
[\lambda I-M]^{-1}x=\lim_{r\to+\infty}\int_0^r e^{-t(\lambda I-M)} xdt.$$

\end{prop}

\begin{rem} The preceding result can be intuitively explained in terms of the 
Laplace transform. When $X=\R$, it  simply establishes that the Laplace transform of the real 
function $f(t):=e^{Mt}$ is given
by $\mathfrak L f(\lambda)= \frac 1{\lambda -M}$ 
for $\Re(\lambda)>\Re (M)$. 
For the general case, it may be observed  that
$$\int_0^r e^{-t(\lambda I - M)}x dt=[\lambda I-M]^{-1}[I-e^{-r(\lambda I - M)}] x.$$
    Thus, the result simply expresses the fact that 
    $e^{-r(\lambda I - M)}\to 0$ as $r\to+\infty$ or, equivalently, that the trivial equilibrium of the system $y'(t)= -(\lambda I-M)y(t)$ is asymptotically stable. 
\end{rem}

\section{Rapid proofs of Theorems \ref{Gronwall abstracto} and \ref{main} } 
\label{dem-main}

With the setting of the preceding section in mind, let us firstly proceed to a proof of 
the abstract  Gronwall inequality. 
\medskip 

\begin{pot} Because $K\in \mathcal L(X)$  is positive, it follows that   $e^{tK}$ is positive. Hence, given $x\in X$ such that $x\geq 0$ we obtain, for arbitrary  $s>\rho_K$:
%$$e^{-st}T_K(t)x\geq 0 \qquad t\geq 0.$$
%Thus, 
$$[sI-K]^{-1}x=\lim_{r\to+\infty}\int_0^r e^{-ts} e^{tK} x\, dt\geq 0.$$
Thus, if 
$sz\le Kz$ then $(sI-K)z\le 0$ and the result is deduced by applying the positive operator $(sI-K)^{-1}$ at both sides of the inequality.
 \end{pot}
\medskip 

%\section{Proof of  Theorem \ref{main}}
 
  Next, we proceed with the maximum 
 principle. 
 \medskip 
 
  \begin{pot2}
    To begin, observe that the case $B= 0$ is straightforward, since 
  (\ref{inecuacion 1})  implies $KLx\le 0$, whence
  $$x=\mathcal Px + KLx \le 0.
  $$
   
Next, assume $B>0$ and apply $K$ at both sides of inequality 
(\ref{inecuacion 1}) to obtain 
 $$
 KLx\leq BKx.
 $$
Because  
  $KLx=x-\mathcal{P}x$, 
  it is deduced that  
 $$x\leq \mathcal{P} x +BK x\leq BK x.$$
 Thus, 
 $$B^{-1}x \le Kx$$
 and the conclusion follows from   
 Theorem  \ref{Gronwall abstracto} and the fact that $B^{-1}>\rho_K$.  \end{pot2}

 \section{New insights about (not so) old results}
\label{extensions}
\label{discrete}

As mentioned in the introduction, the abstract 
formulation of the Gronwall lemma presented in this work allows to give a simple proof, not only of the standard original version but also of many of the extensions that can be found in the literature. For instance, a rather 
popular one is the following, which is valid   for  arbitrary continuous  functions 
$A,  B$ and $C$  with $B, C\ge 0$: if
$$x(t)\le A(t) + C(t)\int_a^t B(s)x(s)ds\qquad a\le t\le b$$
then 
\begin{equation}
    \label{nonconstant A} 
x(t)\le A(t) + C(t) \int_a^t A(s)B(s)e^{\int_s^t B(r)C(r)dr}ds.
\end{equation}
According to our previous setting, here $Kx(t):= C(t)\int_a^t B(s)x(s)ds$ and 
the estimate is optimal because the right-hand side term in (\ref{nonconstant A}) is the (unique) solution of  the fixed point equation $y = A + Ky$. 
When $A\ge 0$, it is often mentioned that the simpler (and perhaps more useful)
estimate  
$$x(t)\le A(t) e^ {C(t) \int_a^t B(s) \, ds}
$$
 also holds, provided that $A$ and $C$ are nondecreasing; 
however, the latter inequality
is unsharp, except for some trivial cases (e.g. $A, C$ constant, or $B\equiv 0$). 
Similarly, we may consider 
the more general operator
$$Kx(t):= \int_a^t k(t,s)x(s)\, ds 
$$
with $k\ge 0$ continuous, which is also quasinilpotent. 
Indeed, due to compactness it suffices to prove that the problem $x = \sigma Kx$  has no nontrivial solutions for arbitrary $\sigma\ne 0$. 
Suppose that $x$ is such a solution and let $t_0\in[a,b]$ be the supremum of those values of $t$ such that $x\equiv 0$ on $[a,t]$. If $t_0<b$, then we may fix $\delta>0$ such 
that $\delta|\sigma|\|k\|_\infty <1$ and $x\not\equiv 0$   on $J:=[t_0,t_0+\delta]$. Thus, for $t\in J$ we have
$$|x(t)| \le |\sigma|\int_{t_0}^t|k(t,s)x(s)|\, ds \le \delta|\sigma|\|k\|_\infty |\|x|_J\|_\infty,
$$
a contradiction. 
This proves that the conclusion of Corollary \ref{y(s,A)} 
with $s=1$ holds, although the exact computation of 
$y=y(1,A)=(I-K)^{-1}A$ is not always possible. 
In \cite{drag, mitri} it is mentioned that  
$$x(t)\le \hat A(t)  e^{
{\int_a^t \hat k(t,s)\, ds }}
$$
where
$$\hat A(t):= \max_{a\le r\le t} A(r), 
\quad \hat k(t,s):= \max_{s\le r\le t} k(r,s).$$
Again, this estimate is not 
optimal, 
although the  
same works refer also to the one obtained in \cite{CM}, which is sharp but not in closed form, namely
$$x(t)\le A(t) + \int_a^t R(t,s)A(s)\, ds,
$$
where 
$$R(t,s):=\sum_{n=1}^\infty k_n(t,s)$$
is the resolvent kernel of $k(t,s)$, 
and $k_n$ are the iterated kernels of 
$k(t,s)$. 
It may be noticed, 
in our abstract  
setting, that the equality $y=A+ Ky$
yields
$$y= A + K(A+Ky)=\ldots = \sum_{j=0}^N K^jA + K^{N+1}y,
$$
so the previous expression is clear because, when 
$Kx(t)=\int_a^tk(t,s)x(s)\, ds$, it is proven by induction that
$$|K^jy(t)|\le \|y\|_\infty\|k\|^n_\infty \frac{(t-a)^n}{n!}. 
$$
For a general quasinilpotent operator $K$, it may not hold that $K^Ny\to 0$, so the formula $y= \displaystyle\sum_{j=0}^\infty K^jA$ cannot 
 be deduced. In any case, if everything else fails,  
 it is always possible to express the value $y(1,A)$ using the non closed-form formula provided by Proposition \ref{resolvent}, that is
 $$y = \int_0^{+\infty} e^{-t(I-K)}A\, dt.
 $$

 To conclude this section, we remark  that the abstract formulation allows immediate extensions of the 
 preceding results to other contexts, 
 in which the  continuity assumptions 
 are relaxed 
 or more general  integral operators are involved. As an example, consider the very recent paper {\cite{gall}}, which is devoted to an inequality with a Stieltjes integral
 \begin{equation}\label{stiel}
 x(t)\le A(t) + \int_a^t x(s)\, dP(s),
  \end{equation}
 where $P$ is a nondecreasing function such that $P(t) < P(t^-) + 1$ for all $t$ and $K$ is  integrable with respect to $dP$. An upper estimate for $x$ is obtained in terms of 
 the generalized exponential function and, furthermore, 
 the estimate turns to an identity if equality in (\ref{stiel}) holds. 
 This fully agrees with our  previous comments regarding the optimality in the Gronwall lemma.

\subsection{Finite dimensional case}

A straightforward application of the preceding results is obtained when $X=\R^n$ in which, for convenience, the operator $K$ shall be identified with its associated matrix. Let us consider in $X$ the coordinate-wise order, that is, the order induced by the cone of those vectors with nonnegative coordinates. It is readily verified that $K$ is positive if and only if all its entries are 
nonnegative and, as  is well known, the value $\rho_K\ge 0$ is the Perron root 
of $K$.   
A matrix version of Corollary \ref{y(s,A)} reads as follows.

\begin{lem}\label{gronwall-Rn}
 Assume that $K\in \R^{n\times n}$  has nonnegative entries 
 and let $A\in \R^n$ and $B>0$ with $B\rho_K<1$. 
 If
 $x\in \R^n$ is such that 
 $$
 x_i\le A_i + B\sum_{j=1}^n k_{ij}x_j
 $$
  for all $i$, then  $x_i\le y_i$ for all $i$, where 
 $$y:= [I-BK]^{-1}A.
 $$
\end{lem}
A particular instance 
corresponds 
to the  discrete Gronwall Lemma, which has many applications in the context of difference equations of the form
$$x_{k+1} - x_k = f(k,x_k) 
$$
for $a\le k< b$, where $a<b$ are integers. 
Specifically, 
\begin{lem}(Discrete Gronwall inequality)
    Assume that 
    $$x_i\le A_i + \sum_{j=a}^{i-1} B_jx_j\qquad a\le i \le b$$
with $B_i\ge 0$ for all $j$. 
Then 
$$
x_i\le A_i + \sum_{j=a}^{i-1} A_jB_j {\prod_{s=j+1}^{i-1}(1+B_s)}\qquad a\le i\le b.$$
\end{lem}
\begin{rem}
As usual, in the last formula it is understood  that the product is equal to  $1$ when $j=i-1$. 
\end{rem}

\begin{proof}
    It suffices to consider, in the previous lemma, $B=1$ and the nilpotent matrix defined as
    $$K:=\left(\begin{array}{ccccc}
        0 & 0 & \ldots & {} & {}\\
        B_1 & 0 & \ldots & {} & {}\\
       B_1 & B_2 & 0 & \ldots & {} \\
        \ldots & \ldots & {} &{}&\\
        B_1 & B_2 & \ldots & B_{n-1} & 0
    \end{array}
    \right). 
    $$
     The computation of $y$ follows by induction or, alternatively, from the method of variation of parameters  applied to the discrete linear system given by $Y_a=A_a$ and $Y_{k} = A_{k}+ \displaystyle\sum_{j=a}^{k-1} B_jY_j$ for $k>a$.     
\end{proof}

\begin{rem}
As before, the previous proof relies on the fact that $\rho_K=0$ and can be easily extended to 
prove an 
exact analog of (\ref{nonconstant A}) in the discrete context:
if 
    $$x_i\le A_i + C_i\sum_{j=a}^{i-1} B_jx_j\qquad a\le i \le b$$
with $B_i, C_i\ge 0$ for all $j$, then
$$
x_i\le A_i + C_i\sum_{j=a}^{i-1} A_jB_j {\prod_{s=j+1}^{i-i}(1+B_sC_s)}\qquad a\le i\le b.$$     
It is said in \cite[Ch. XIV]{mitri} that 
the latter inequality is the best possible and, again, this comment agrees with the above discussion. 
Several other results for discrete Gronwall-type inequalities may be found in the works of J. Popenda, as recalled in \cite{schm}. 

\end{rem}
 
 \section{A non-Fredholm example: strong maximum principle for $-\Delta$}
\label{lapl} 

In this section, we shall attempt to connect the Gronwall inequality with the well-known maximum principle for elliptic operators. For simplicity, 
our study is restricted to a simple equation in a smooth domain, although much more general results can be obtained. 

Let  $\Omega\subset\R^n$ 
be a smooth bounded domain and consider $X:=C(\overline \Omega)$ 
with the standard uniform norm. Set 
$$L:=-\Delta: D\subset X\to X,$$
where the domain $D$ shall be taken as the set of those functions $x\in C^{1}(\overline\Omega)$ such that $\Delta x$ exists and can be extended
continuously to $\overline\Omega$,
% $$D:= \lbrace x\in X:    x\vert_{\Omega}\in C^{2}(\Omega) \text{ and } \Delta x \text{ can be extended continuously to }\partial \Omega \rbrace. $$
The kernel of $L$ is the 
subspace of $D$ of those functions that are harmonic in $\Omega$.  
Let us consider the closed subspace of those functions that vanish on the boundary, namely
$$Z=\lbrace x\in C(\overline \Omega): \ x|_{\partial \Omega }=0\rbrace.$$ 
As is well known, the weak maximum principle for subharmonic functions states 
that if $x\in \ker(L)$, then its maximum  value over $\overline \Omega$ is achieved on $\partial\Omega$. 
Moreover, $L$ has a compact right inverse $K:X\to Z\subset X$ given by
$$Kx(t)=\int_\Omega G(t,s)x(s)\, ds,
$$
where $G$ is the Green function associated to the homogeneous Dirichlet problem, that is, if $z=Kx$ then
$$-\Delta z= x,\qquad z|_{\partial \Omega}=0.
$$
The weak maximum principle implies that $G\ge 0$, that is, $K$ is a positive operator. 
Here, $K$ is also positive in the  
$L^2$ sense; more precisely,   the 
set of eigenvalues  of $K$ consists in a sequence of real values $\mu_n\searrow 0$ and the first one is simple. The value $\lambda_1:=\frac 1{\mu_1}>0$ is referred in the literature as the first eigenvalue of $-\Delta$. Thus, in the context of Theorem \ref{Gronwall abstracto}, the condition $s>\rho_K$  simply reads $B<\lambda_1$, where $B:=\frac 1s$.

 \begin{prop} (Maximum principle)
Let $x\in D$ be such that  
$$
-\Delta x \leq Bx,
\qquad x|_{\partial \Omega} \leq 0
 $$
 with $0\le B<\lambda_1$. Then $x\leq 0$.
\end{prop}
\begin{proof} When $B= 0$ the result is trivial.
For $B>0$, observe that $\mathcal P x\in \ker(L)$ 
and, because $K:X\to Z$, it follows that   $\mathcal{P}x|_{\partial \Omega}=x|_{\partial \Omega}\le 0$. This implies $\mathcal{P}x\le 0$
and the conclusion follows then from Theorem 
   \ref{main}.
\end{proof}

\begin{rem}
It is clear that, 
unlike the Gronwall Lemma,  
    the latter result is still valid when $B<0$. 
    However, 
    in this case the proof  
    follows from a direct argument and not as a consequence of Theorem
    \ref{Gronwall abstracto}. 
        More generally, the constant $B$ may be replaced by a function $b\in C(\overline \Omega)$ such that $b(s) <\lambda_1$ for all $s\in\overline\Omega$. 
    
\end{rem}

\section{Uniqueness and continuous dependence}
\label{uniq}

As mentioned in the introduction, the Gronwall lemma is commonly employed for proving uniqueness for  problems of different kinds. This usually requires a Lipschitz condition on the nonlinear terms of the equation. So a natural question arises: how can this condition  be extended to the abstract setting? 
It seems reasonable to consider the following:

\begin{defn} 
Let $X$ be a Banach lattice and let  $N:\Omega\to X$ be an operator, where $\Omega$ is an open subset of $X$. We shall say that $N$ is  Lipschitz 
in the lattice sense if there exists a constant $C>0$ such that 
$$|N(x)-N(y)|\leq C |x-y|$$ for all $x,y\in \Omega$. 
\end{defn}

The following lemma will be useful for our purposes. 

\begin{lem} Let $X$ be a Banach lattice and let $M\in \mathcal L(X)$ be positive. Then $|Mx|\le M |x|$ for all $x\in X$. 
\end{lem}
\begin{proof}
Since $x, -x\le |x|$, it follows that
$$Mx\le M|x|, \qquad  -Mx=M(-x) \le M|x|. 
$$
By definition of supremum, we conclude:
$$|Mx| =Mx\vee - Mx  \le M|x|.
$$ 
\end{proof}

We are now in condition of establishing  a general  result: 
\begin{thm}
Consider the problem 
\begin{equation}
\label{unic}
Lx=N(x),\qquad \mathcal{P}x=x_0
\end{equation}
  where 
 $L$ and $\mathcal{P}$ are like in 
 Theorem \ref{main} and 
 $N:\Omega\subset X\to X$ is Lipschitz in the lattice sense with constant $C>0$ satisfying 
 $C\rho_K<1$ and $x_0\in \ker(L)$. 
 Then (\ref{unic}) has at most one solution in $D\cap \Omega$. 
 Furthermore, if $K$ is compact, then the mapping $x_0\mapsto x$ is continuous and its domain is open in $\ker (L)$.
\end{thm}
\begin{proof}
Suppose that  $x_1, x_2\in D\cap\Omega$ are solutions of (\ref{unic}), then
$$x_i=x_0 + KN(x_i).
$$
Hence 
$$x_1-x_2=K[N(x_1)-N(x_2)]$$
and, from the preceding lemma, we deduce:
$$|x_1-x_2|\leq K|N(x_1)-N(x_2)|\le CK|x_1-x_2|. $$
Set $z:=|x_1-x_2|$, then  
$$C^{-1}z\le Kz$$
and from Theorem   \ref{Gronwall abstracto} we conclude that $z\leq 0$. Thus, Proposition \ref{modulo igual a cero} yields $z=0$. 

Next, assume that $x$ and $\hat x$ are solutions corresponding to some values 
$x_0, \hat x_0 \in \ker(L)$. As before, it is deduced that 
$$x-\hat x=x_0-\hat x + K(N(x)-N(\hat x)),
$$
whence 
 $$|x-\hat x|\leq |x - \hat x_0 | +  CK|x-\hat x|, $$
that is 
$$|x-\hat x| \le [I-CK]^{-1}|x_0 - \hat x_0|
$$
and the continuous dependence follows trivially from the continuity of the operator $[I-CK]^{-1}$, together with 
Proposition \ref{equiv-norm}.  

To conclude, 
 assume that $K$ is compact and that $\hat x$ is a solution for a certain $\hat x_0$. We want to solve the problem for $x_0$ close enough to $\hat x_0$ or, equivalently, to find $x$ such that $x=x_0 + KN(x)$. This may be transformed in the following fixed point problem:
$$x = \hat x + x_0 - \hat x_0 + K(N(x)-N(\hat x)):= F(x),
$$
where the operator $F:\Omega\to X$ is well defined and 
compact. 
Fix $r>0$ such that  
$\overline B_r(\hat x)\subset \Omega$, where $B_r(\hat x)$ is the open ball of radius $r$ centered at $\hat x$. 
For $\lambda\in [0,1]$, let $F_\lambda(x):= \hat x + \lambda\left[x_0 - \hat x_0 + K(N(x)-N(\hat x))\right]$ and assume that $x= F_\lambda (x)$ for some $x\in \partial  B_r(\hat x)$. 
Again, it follows that
$$|x-\hat x| \le \lambda(|x_0-\hat x_0| + CK|x-\hat x|) \le |x_0-\hat x_0| + CK|x-\hat x|,
$$
so 
$$|x-\hat x|\le [I-CK]^{-1}|x_0-\hat x_0|.$$
From Proposition \ref{equiv-norm}, this implies
$$r=\|x-\hat x\|\le \eta\|x_0-\hat x_0\|,
$$
where $\eta$ is the norm of the operator $[I-CK]^{-1}$. 
It  is deduced that $F_\lambda$ does not have fixed points on $\partial B$ when
$\|x_0-\hat x_0\| < \frac r\eta$ and, from the homotopy invariance of the Leray-Schauder degree,
$$\deg(I-F,B_r(\hat x),0)= \deg(I-F_0, B_r(\hat x),0) = \deg(I,\overline B_r(\hat x),\hat x)=1. 
$$
This implies that $F$ has a fixed point in $B_r(\hat x)$, and so concludes the proof.  
 \end{proof}

\section{Concluding remarks}
\label{disc}
 
As a consequence of the previous 
section, if the Lipschitz assumption is global, namely $\Omega=X$, then 
the domain of the map $x_0\mapsto x$ associated to problem 
(\ref{unic})   is either empty or $\ker(L)$. Indeed, it suffices to observe, 
in the previous proof, that  the choice of $r$ is arbitrary and guarantees the existence of 
a solution for $x_0$ in a ball of radius $\frac r\eta$ centered at $\hat x_0$. 

It is noticed,  in the particular case (\ref{ivp}), 
that such a global condition guarantees the 
existence of a unique solution for arbitrary $x_0$, without any restriction on the 
Lipschitz constant 
$C$. At first sight, this might look striking: we know that the size of the 
constant is not 
relevant for the uniqueness, because $\rho_K=0$, but
the fixed point operator $T(x):=x_0 + KN(x)$ is not a contraction unless the Lipschitz  constant 
is small. 
Incidentally, observe that 
property (\ref{normal}) becomes very important here, since the previously deduced inequality 
$$|x_1-x_2|\le CK |x_1-x_2|,
$$
together with Proposition \ref{equiv-norm} imply
$$\|x_1-x_2\|\le C \|K |x_1-x_2|\| \le C\|K\| \|x_1-x_2\|,
$$
so $T$ is a contraction when $C\|K\|<1$. 
Fortunately, the latter restriction is not needed at all: 
as is well known,  a standard strategy consists in proving that local solutions can be glued together until the whole interval $[a,b]$ is covered. 

Certainly, the preceding argument cannot be extended to  the abstract setting, since 
 there is no such a thing we may  call a ``local solution". 
 This issue is hidden in (\ref{ivp}) due to the crucial fact that, 
 roughly speaking, the operators $K, L$ and $N$ commute 
 with the restriction operator. 
 This is why the fixed point equation $x=T(x)$ looks the same
 when considered over $[a,a+\delta]$ instead 
 of the whole interval $[a,b]$. 
 Thus, an attempt to extend the idea of local solution to an abstract context 
 could be done 
 by defining Banach spaces $X_\delta$ and operators $r_\delta:X\to X_\delta$, $K_\delta, N_\delta:X\to X$ such that $r_\delta KN=K_\delta N_\delta r_\delta$, so the original equation, when translated into the space $X_\delta$, takes the form
 $$z= r_\delta x_0 + K_\delta N_\delta z.
 $$
In the standard case (\ref{ivp}), 
the proof of local existence takes advantage 
of the fact that the norm of the associated integral operator, 
 regarded as an endomorphism of $X_\delta$, tends to $0$ as $\delta\to 0$; thus,
  choosing an appropriate value of $\delta>0$ it is immediately verified  that the operator $T_\delta:=r_\delta + K_\delta N_\delta$  is a contraction

 In contrast with the previous situation, 
uniqueness results for the equation $-\Delta x=f(t,x)$ under Dirichlet conditions  usually require a  global Lipschitz condition with small constant. This is due to the fact that, here,   $\rho_K>0$.  
It is well known that the Lipschitz condition may 
be replaced by a one-side 
growth condition on $f$, namely
$$[f(t,x)-f(t,y)](x-y) \le c(x-y)^2
$$
for some constant $c < \lambda_1$. However, this latter assumption makes use of the Hilbert structure: if $x$ and $y$ are solutions, then 
$$\langle -\Delta (x-y),x-y\rangle =\langle f(\cdot,x)-f(\cdot,y),x-y\rangle \le c\|x-y\|^2,
$$
from where it is deduced that $\|x-y\|^2\le \frac c{\lambda_1}\|x-y\|^2$, that is, 
$x=y$.

\section*{Declarations}

{\bf Funding and/or Conflicts of interests/Competing interests}
\medskip

\noindent This work was  supported by the projects CONICET PIP 11220130100006CO and UBACyT 20020160100002BA.
The authors declare that there are no  conflicts of interests or competing interests.


\begin{thebibliography}{1}

\bibitem{AT}
R. P. Agarwal and E. Thandapani, 
\textit{Remarks on Generalizations of Gronwall's Inequality},  Chinese Journal of Mathematics 
9, No. 2 (1981),
1--22. 


\bibitem{cha}
J. Chandra and B.  Fleishman, \textit{On a generalization of the Gronwall-Bellman
lemma in partially ordered Banach spaces}, J. Math. Anal. Appl. 31 (1970), 668--681.

\bibitem{CM}
 S. Chu and F. Metcalf, \textit{On Gronwall’s Inequality}, Proc. Amer.
Math. Soc. 18 (1967), 439--440.


\bibitem{drag} S. Dragomir, 
\textit{Some Gronwall Type Inequalities and Applications},
Science Direct Working Paper No S1574-0358(04)70847-3 (2003), 1--197.
\bibitem{erbe} 
L. Erbe, Q. Kong, \textit{Stieltjes integral inequalities of Gronwall type and applications}, Annali di Matematica pura ed applicata 157 (1990), 77--97.

 \bibitem{gall} 
 C. Gallegos, I M\'arquez Alb\'es, A. Slavík,
\textit{A general form of Gronwall inequality with Stieltjes integrals},
Journal of Mathematical Analysis and Applications,
 541 No. 1 (2025), 128674.


\bibitem{gro}
T. H. Gronwall, \textit{Note on the derivatives with respect to a parameter of the solutions of a system of differential equations}, Ann. of Math. (2) 20  (1919), 292--296.

\bibitem{ming} A. Mingarelli,
\textit{On a Stieltjes Version of Gronwall's Inequality},   Proceedings of the American Mathematical Society 82, No. 2 (1981), 249--252.


\bibitem{mitri}
D. Mitrinovi\'c, J. Pe\v cari\'c and A. Fink, 
\textit{Inequalities involving functions
and their integrals
and derivatives}. Springer Dordrecht, 1991. 



\bibitem{Na}
R. Nagel, (ed.), \textit{One-parameter semigroups of positive operators}, Lecture Notes in Math.
1184, Springer-Verlag, Berlin-Heidelberg-New York-Tokyo, 1986.

\bibitem{ro}
V. Romanov, 
\textit{A generalized Gronwall–Bellman inequality, 
}
Siberian Mathematical Journal  61 No. 3 (2020) 532--537.


\bibitem{schm}
{E. Schmeidel, M. Migda and E. Magnucka-Blandzi}, 
\textit{Mathematical Works Of Jerzy Popenda},
Applied Mathematics E-Notes 2 (2002), 155--170.  

\bibitem{ye} 
H. Ye, J. Gao, Y. Ding, \textit{A generalized Gronwall inequality and its application
to a fractional differential equation},
J. Math. Anal. Appl. 328 (2007) 1075--1081.


\end{thebibliography}
\end{document}